\documentclass[10pt]{amsart}
\usepackage[matrix,arrow]{xy}
\usepackage{amssymb}
\usepackage{amsmath}
\usepackage{amsfonts}
\usepackage{epsfig}
\usepackage{color}
\usepackage{epstopdf}
\usepackage{graphicx}
\usepackage{amsthm}
\usepackage{enumerate}
\usepackage[mathscr]{eucal}
\usepackage{verbatim}
\usepackage{bm}

\usepackage[bookmarks=true,hyperindex,pdftex,colorlinks,citecolor=blue,
linkcolor=blue,urlcolor=blue]{hyperref}
\usepackage{tikz}
\usetikzlibrary{matrix,arrows.meta}

\theoremstyle{plain}
\newtheorem{theorem}{Theorem}[section]
\newtheorem{cor}[theorem]{Corollary}
\newtheorem{prop}[theorem]{Proposition}
\newtheorem{lemma}[theorem]{Lemma}

\theoremstyle{definition}

\newtheorem{example}[theorem]{Example}

\newtheorem{question}[theorem]{Question}

\DeclareMathOperator{\re}{Re}

\renewcommand{\subset}{\subseteq}

\title[Geometry of the space of compact operators]
{Geometry of the space of compact operators endowed with the numerical radius norm}

\author[M.~Han]{Manwook Han}
\address[M.~Han]{Department of Mathematics, Chungbuk National University, Cheongju, Chungbuk 28644, Republic of Korea}
\email{\texttt{mwhan0828@gmail.com}}

\author[S.~K.~Kim]{Sun Kwang Kim}
\address[S.~K.~Kim]{Department of Mathematics, Chungbuk National University, Cheongju, Chungbuk 28644, Republic of Korea\newline
	\href{http://orcid.org/0000-0002-9402-2002}{ORCID: \texttt{0000-0002-9402-2002}  }}
\email{\texttt{skk@chungbuk.ac.kr}}

\thanks{}


\begin{document}

\begin{abstract}We investigate the space of bounded linear operators on a Banach space equipped with a norm which is equivalent to the operator norm such that the subspace of compact operators is an M-ideal. In particular, we observe that the space of compact operators on $\ell_p$ $(1<p<\infty)$ equipped with the numerical radius norm is an M-ideal whenever the numerical index of $\ell_p$ is not $0$. On the other hand, we show that the space of compact operators on a Banach space containing an isomorphic copy of $\ell_1$ whose numerical index is greater than $1/2$ is not M-ideals. We also study the proximinality, the existence of farthest points and the compact perturbation property for the numerical radius.
\end{abstract}

\maketitle
\section{Introduction}
For a Banach space $X$ and its closed subspace $J\subset X$, we say that $J$ is an M-ideal in $X$ if $X^*=J^*\oplus_1 J^\perp$ where $J^\perp$ is the annihilator of $J$.  This notion is firstly introduced by Alfsen and Effros in \cite{AE},  and afterwards many authors obtained fruitful observations on the space of compact operators. For a Banach space $X$, we denote the space of bounded linear operators from $X$ to itself by $\mathcal{L}(X)$, and the subspace of  compact operators in $\mathcal{L}(X)$ by $\mathcal{K}(X)$. In 1950, Dixmier observed that $\mathcal{K}(\ell_2)$ is an M-ideal in $\mathcal{L}(\ell_2)$ (see \cite{Dix}), and it is the first non-trivial example of M-ideals in the space of linear operators.  In 1979, Lima proved that $\mathcal{K}(\ell_p)$ is an M-ideal in $\mathcal{L}(\ell_p)$ for $1<p<\infty$, and that $\mathcal{K}(c_0)$ is an M-ideal in $\mathcal{L}(c_0)$ (see \cite{L1}). A few years later, Harmand and Lima showed that for any separable Banach space $X$ if $\mathcal{K}(X)$ is an M-ideal in $\mathcal{L}(X)$, then $X^*$ has the metric compact approximation property (see \cite{HL}). This result was extended to the non-separable spaces by Oja who demonstrated that M-ideals of compact operators are separably determined (see \cite{Oja1}). For more information on the theory of M-ideals, we refer the reader to the classic monograph \cite{HWW}.

In this paper, we focus on the spaces $\mathcal{L}_N(X)$ and $\mathcal{K}_N(X)$ which are the space of linear operators and the space of compact operators equipped with a norm $N$ which is equivalent to the operator norm. One of the most typical example of such norms can be derived by the numerical radius. For a set
$$\Pi(X)=\{(x^*,x)\in S_{X^*}\times S_{X}:x^*(x)=1\},$$
the numerical radius $v(T)$ of  $T\in\mathcal{L}(X)$ is defined by
$$v(T)=\sup\{|x^*Tx|\in\mathbb{K}:(x^*,x)\in\Pi(X)\}.$$
The numerical index $n(X)$ of $X$ is the infimum of numerical radius of norm one operators. Indeed,  
$$n(X)=\inf\{v(T):\|T\|=1\}.$$

It is obvious that $v$ is a seminorm on $\mathcal{L}(X)$, and that if $n(X)>0$, then $v$ is equivalent to the operator norm. In particular, if $n(X)=1$, then the numerical radius and operator norm are the same. It is worth to note that $n(X)>0$ whenever $X$ is a complex Banach space. We refer the reader to \cite{BD1, BD2} for more details on numerical radius and numerical index.

We first observe that $\mathcal{K}_v(\ell_p)$ is an M-ideal in $\mathcal{L}_v(\ell_p)$ for $1<p<\infty$ whenever the numerical index of $\ell_p$ is not $0$. Secondly, we obtain some characterizations of $\mathcal{L}_N(X)$ such that $\mathcal{K}_N(X)$ is an M-ideal. These characterizations are the analogy of the following result of Werner who characterized $\mathcal{L}(X)$ such that $\mathcal{K}(X)$ is an M-ideal. In the statement, $\|S\|_e$ for $S\in \mathcal{L}(X)$ is the essential norm which is the distance from $S$ to $\mathcal{K}(X)$. 

\begin{theorem}\label{eqvalence11}\cite{Wer}
For  a Banach space $X$, the following are equivalent.
\begin{enumerate}
\item $\mathcal{K}(X)$ is an M-ideal in $\mathcal{L}(X)$.
\item For each $T\in\mathcal{L}(X)$, there is a net $(K_\alpha)$ in $\mathcal{K}(X)$ such that $K_\alpha^*$ converges to $T^*$ strongly and that
$$\limsup_\alpha\|S+T-K_\alpha\|\leq\max\{\|S\|,\|S\|_e+\|T\|\} \text{~for every~} S\in\mathcal{L}(X).$$
\item For each $T\in\mathcal{L}(X)$, there is a bounded net $(K_\alpha)$ in $\mathcal{K}(X)$ such that $K_\alpha^*$ converges to $T^*$ strongly and that
$$\limsup_\alpha\|S+T-K_\alpha\|\leq\max\{\|S\|,\|T\|\} \text{~for every~} S\in\mathcal{K}(X).$$
\end{enumerate}
\end{theorem}

Using the aforementioned characterizations, we derive an upper bound of the numerical index of a Banach space $X$ whenever $\mathcal{K}_v(X)$ is an M-ideal in $\mathcal{L}_v(X)$. As a corollary, we show that if $X$ is a Banach space such that $n(X)>1/2$ and $\mathcal{K}_v(X)$ is an M-ideal in $\mathcal{L}_v(X)$, then it is Asplund. 

We also investigate the geometry of such M-ideals of compact operators. Under the assumptions that  $n(X)>0$ and $\mathcal{K}_v(X)$ is an M-ideal in $\mathcal{L}_v(X)$, we get that every extreme point of $B_{\mathcal{K}_v(X)^*}$ is of the form $\lambda x^{**}\otimes x^*$ where $(x^{**},x^*)\in\Pi(X^*)$ and $\lambda\in\mathbb{T}$. This leads us to see that if the adjoint $T^*$ of $T\in\mathcal{L}_v(X)$ does not attain its numerical radius, then $T$ does not have a farthest point in $B_{\mathcal{K}_v(X)}$ and it holds that
$$v(T)=\inf_{K\in\mathcal{K}_v(X)}v(T-K).$$
 The latter one shows that the set of operators whose adjoint operators do not attain their numerical radius is nowhere dense in $\mathcal{L}_v(X)$ under these assumptions.

\section{Results}

\subsection{Examples of proper M-ideals}~

As it is shown that $\mathcal{K}(\ell_p)$ for $1<p<\infty$ is a proper M-ideal, an M-ideal which is not an M-summand in $\mathcal{L}(\ell_p)$, we show that $\mathcal{K}_v(\ell_p)$ is also a proper M-ideal in $\mathcal{L}_v(\ell_p)$ whenever $n(\ell_p)>0$. In order to see this, we present the well known 3-ball property which is the geometric characterization of an M-ideal.

\begin{theorem}\label{3ball}\cite[I.2. Theorem 2.2]{HWW}
For a closed subspace $J$ of a Banach space $X$, the followings are equivalent.
\begin{enumerate}
\item $J$ is an M-ideal in $X$.
\item (The 3-ball property) If $x_1,x_2,x_3\in X$ and $r_1,r_2,r_3>0$ satisfies 
$$\bigcap_{i=1}^3 B(x_i,r_i)\neq\emptyset~\text{ and }~B(x_i,r_i)\cap J\neq\emptyset \text{~for~} i=1,2,3,$$
then
$$\bigcap_{i=1}^3B(x_i,r_i+\varepsilon)\cap J\neq\emptyset \text{~for every~} \varepsilon>0.$$
\item (The restricted 3-ball property) For every $y_1,y_2,y_3\in B_J$, $x\in B_X$ and $\varepsilon>0$, there is $y\in J$ such that
$$\|x+y_i-y\|\leq1+\varepsilon \text{~for~}  i=1,2,3.$$
\end{enumerate}
\end{theorem}

\begin{example}\label{basicexample}For a sequence of finite dimensional spaces $\{F_i\}$, let $X_1$ be $\ell_p(F_i)$ for $p\in (1,\infty)$ and $X_2$ be  $c_0(F_i)$,  the $\ell_p$-sum and the $c_0$-sum of $\{F_i\}$, respectively. For each $k\in\{1,2\}$, if $n(X_k)>0$, then  $\mathcal{K}_v(X_k)$ is an M-ideal in $\mathcal{L}_v(X_k)$.
\end{example}

\begin{proof} 
 In order to show that $\mathcal{K}_v(X_k)$ is an M-ideal, we prove that it has the restricted 3-ball property of Theorem \ref{3ball}. 

Let $\varepsilon>0$, $S_1, S_2, S_3\in B_{\mathcal{K}_v(X_k)}$ and $T\in B_{\mathcal{L}_v(X_k)}$ be given. For the canonical projection $P_j\in\mathcal{L}(X_k)$ to the first $j$th-coordinates,  there is $n\in\mathbb{N}$ such that $v(P_nS_iP_n-S_i)<\varepsilon$ for each $i=1,2,3$. Indeed, since the sequences $\{P_j\}$ and $\{P_j^*\}$ converge strongly to the corresponding identity operators respectively, for each compact operator $K$ there is $m\in\mathbb{N}$ such that
\begin{align*}
v(P_jKP_j-K)&\leq\|P_jKP_j-K\|\\
&\leq\|P_jKP_j-P_jK\|+\|P_jK-K\|\\
&=\|(id_{X_k}-P_j)^*K^*P_j^*\|+\|(id_{X_k}-P_j)K\|<\varepsilon
\end{align*}
 for every $j\geq m$. 

 We now claim that
$$v(T+P_nS_iP_n-(P_nT+TP_n-P_nTP_n))\leq 1.$$

For $\overline{P_n}=id_{X_k}-P_n$, it holds that 
$$v(T+P_nS_iP_n-(P_nT+TP_n-P_nTP_n))=v(\overline{P_n}T\overline{P_n}+P_nS_iP_n),$$
and that 
$$x^*(\overline{P_n}x)+x^*(P_nx)=1$$
for each pair $(x^*,x)\in\Pi(X_k)$. Moreover, we have 
$$\|\overline{P_n}^*x^*\|\|\overline{P_n}x\|=1-\lambda \text{~and~} \|P_n^*x^*\|\|P_nx\|=\lambda$$ for $\lambda=x^*(P_nx)$ from the H\"{o}lder's inequality. 
Indeed, we have
\begin{align*}
1&=x^*(\overline{P_n}x)+x^*(P_nx)\\
&=x^*(\overline{P_n}\, \overline{P_n}x)+x^*(P_nP_nx)\\
&=\overline{P_n}^*x^*(\overline{P_n}x)+P_n^*x^*(P_nx)\\
&\leq\|\overline{P_n}^*x^*\|\|\overline{P_n}x\|+\|P_n^*x^*\|\|P_nx\|.
\end{align*}

For $k=1$, we see that \begin{align*}
\|\overline{P_n}^*x^*\|\|\overline{P_n}x\|+\|P_n^*x^*\|\|P_nx\|&\leq \|(\|\overline{P_n}^*x^*\|,\|P_n^*x^*\|)\|_{p^*}\|(\|\overline{P_n}x\|,\|P_nx\|)\|_p=\|x^*\|\|x\|=1
\end{align*}
where $\|(\cdot,\cdot)\|_s$ for $s\in [1,\infty)$ is the $s$-norm on $\mathbb{R}^2$ and $p^*$ is the conjugate of $p$.

For $k=2$, we see that 
\begin{align*}
\|\overline{P_n}^*x^*\|\|\overline{P_n}x\|+\|P_n^*x^*\|\|P_nx\|&\leq \|(\|\overline{P_n}^*x^*\|,\|P_n^*x^*\|)\|_{1}\|(\|\overline{P_n}x\|,\|P_nx\|)\|_\infty=\|x^*\|\|x\|=1
\end{align*}
where $\|(\cdot,\cdot)\|_\infty$ is the supremum norm on $\mathbb{R}^2$. Hence, we get the desired equalities. 

These show that, whenever $0<\lambda<1$, 
\begin{align*}
&|x^*(\overline{P_n}T\overline{P_n}x)|=(1-\lambda)\left|\frac{\overline{P_n}^*x^*}{\|\overline{P_n}^*x^*\|}\left(T\frac{\overline{P_n}x}{\|\overline{P_n}x\|}\right)\right|\leq(1-\lambda) v(T)\leq1-\lambda \quad \text{and}\\
&|x^*\left(P_nS_iP_nx\right)|=\lambda\left|\frac{P^*_nx^*}{\|P^*_nx^*\|}\left(S_i\frac{P_nx}{\|P_nx\|}\right)\right|\leq\lambda v(S_i)\leq \lambda.
\end{align*}

Hence, we deduce 
$$|x^*\left((\overline{P_n}T\overline{P_n}+P_nS_iP_n)x\right)|\leq1.$$

When it holds $\lambda=0$ or $\lambda=1$, it is easier to get the same inequality. Consequently, we obtain
$$v(T+P_nS_iP_n-(P_nT+TP_n-P_nTP_n))=v(\overline{P_n}T\overline{P_n}+P_nS_iP_n)\leq1.$$

Therefore, we see
\begin{align*}
v(T+S_i-&(P_nT+TP_n-P_nTP_n))\\
&\leq v(T+P_nS_iP_n-(P_nT+TP_n-P_nTP_n))+v(P_nS_iP_n-S_i)\\
&<1+\varepsilon
\end{align*}
which completes the proof. 
\end{proof}

 It is known that $n(\ell_p)=0$ only if $p=2$ and $\ell_p$ is a real Banach space (see \cite{MMP}). Hence, $\mathcal{K}_v(\ell_p)$ is an M-ideal in $\mathcal{L}_v(\ell_p)$ for $p\in (1,\infty)$ in the complex case and for $p\in (1,2)\cup(2,\infty)$ in the real case.
 For a Banach space $X\in \{c_0,~\ell_1,~\ell_\infty\}$, it is well known that $n(X)=1$ (see \cite{MP}) which implies $\mathcal{L}(X)$ and $\mathcal{K}(X)$ are isometric to $\mathcal{L}_v(X)$ and $\mathcal{K}_v(X)$, respectively. Since $\mathcal{K}(X)$ is an M-ideal in $\mathcal{L}(X)$ only if $X=c_0$ (see \cite[Chap. VI]{HWW}), the same statements hold for the numerical radius. It is worth to note that, for $c_0(F_i)$ of Example \ref{basicexample}, the equality $n(c_0(F_i))=\inf_i n(F_i)$ holds (see \cite{MP}). Hence, the fact that $\mathcal{K}_v(c_0(F_i))$ is an M-ideal whenever $n(c_0(F_i))>0$ can not be obtained from the classical results of M-ideals of compact operators endowed with the operator norm.

To see the properness of M-ideals in Example \ref{basicexample}, we give the numerical radius version of \cite[VI.4. Proposition 4.3]{HWW}.

\begin{prop}
For a Banach space $X$ such that $n(X)>0$, if $\mathcal{K}_v(X)$ is an M-summand in $\mathcal{L}_v(X)$, then $\mathcal{K}_v(X)=\mathcal{L}_v(X)$.
\end{prop}
\begin{proof}
Assume that $\mathcal{L}_v(X)=\mathcal{K}_v(X)\oplus_\infty \mathcal{R}$ for a non-trivial subspace $\mathcal{R}$ of $\mathcal{L}_v(X)$. Let $T$ be an operator in ${\mathcal{R}}$ whose numerical radius is $1$, and $\varepsilon\in (0,1)$ be given. For a pair $(x^*_0,x_0)\in\Pi(X)$ such that $|x^*_0(Tx_0)|\geq 1-\varepsilon$, take a number $\lambda$ whose modulus is $1$ such that $\lambda x_0^*(Tx_0)=|x_0^*(Tx_0)|$. 

It is obvious that $K=\bar{\lambda} x_0^*\otimes x_0$ is a compact operator whose numerical radius is $1$. Hence, we get $v(T+K)=1$ from the assumption. However, we have $$v(T+K)\geq|x_0^*\left((T+K)x_0\right)|\geq2-\varepsilon$$ which is a contradiction.
\end{proof}

$~$

\subsection{M-embeddings}~

We say that a Banach space $X$ is an M-embedded space if $X$ is an M-ideal in $X^{**}$, and many spaces like $c_0$, the little Bloch space $B_0$, and the space of compact operators on a Hilbert space are known to be M-embedded (see \cite[III.1. Example 1.4]{HWW}). Especially, it is known that whenever $X$ is a reflexive Banach space if $\mathcal{K}(X)$ is an M-ideal in $\mathcal{L}(X)$, then $\mathcal{K}(X)^{**}=\mathcal{L}(X)$ which implies $\mathcal{K}(X)$ is M-embedded (see \cite[VI.4. Proposition 4.11]{HWW}). In the present section, we observe the same result for the case of $\mathcal{K}_v(X)$. In order to prove the analogue of \cite[VI.4. Proposition 4.11]{HWW}, we show the existence of a shrinking compact approximation for the identity on the numerical radius setting.

For an operator $T\in\mathcal{L}(X)$, if a net $\{K_\alpha\}\in\mathcal{K}(X)$ converges strongly to $T$, then we say that $\{K_\alpha\}$ is a compact approximation for $T$. Additionally, if $K_\alpha^*$ also converges strongly to $T^*$, then $\{K_\alpha\}$ is said to be a shrinking compact approximation for $T$.

\begin{prop}\label{scap}
Let $X$ be a Banach space with $n(X)>0$. If $\mathcal{K}_v(X)$ is an M-ideal in $\mathcal{L}_v(X)$, then there is a shrinking compact approximation $\{K_\alpha\}\subset B_{\mathcal{K}_v(X)}$ for the identity map $id_X$.
\end{prop}
\begin{proof}
This is an analogy of \cite[VI.4. Proposition 4.10]{HWW}, but we give details for the completeness. Note that $B_{\mathcal{K}_v(X)}$ is $\sigma(\mathcal{L}_v(X),\mathcal{K}_v(X)^*)$-dense in $B_{\mathcal{L}_v(X)}$ (see  \cite[I.1. Remark 1.13]{HWW}). Since $v(id_X)=1$ and the functional $x^{**}\otimes x^*$ for $x^{**}\in X^{**}$ and $x^*\in X^*$ defined by $x^{**}\otimes x^*(K)= x^{**}(K^*x^*)$ for each $K\in \mathcal{K}_v(X)$ belongs to $\mathcal{K}_v(X)^*$, there is a net $\{L_\beta\}_{\beta\in I}\subset B_{\mathcal{K}_v(X)}$ of compact operators such that $L_\beta^*$ converges to $id_{X^*}$ in the weak operator topology. After passing to the convex combination, we obtain the desired shrinking compact approximation. Indeed, define a directed set $J$ by
$$J=\{\gamma=(U,\beta)~:~U\text{~is a SOT-open neighborhood of~}id_{X^*},\beta\in I\}$$
with the canonical order $\prec$ defined by $(U_1,\beta_1)\prec(U_2,\beta_2)$ whenever $U_1\supseteq U_2$ and $\beta_1\leq\beta_2$. For the convenience, here we denote the weak operator topology and the strong operator topology by WOT and SOT. It is well known that WOT-closure of a convex subset of $\mathcal{L}(X)$ coincides with its SOT-closure. Hence, $id_{X^*}\in \overline{\text{conv}\{L_\beta^*~:~\beta \geq \beta_0\}}^{\text{WOT}}=\overline{\text{conv}\{L_\beta^*~:~\beta \geq \beta_0\}}^{\text{SOT}}$ for every $\beta_0\in I$. Take $T_{\gamma_0}$ for each $\gamma_0=(U_0,\beta_0)\in J$ such that
$$T_{\gamma_0}^*\in \text{conv}\{L_\beta^*~:~\beta\geq\beta_0\}\cap U_0.$$
It is obvious that $T_{\gamma}^*$ converges to $id_{X^*}$ in SOT, and this implies $T_\gamma$ converges to $id_X$ in WOT. By repeating the same argument for $\{T_\gamma\}$, we obtain the desired net $\{K_\alpha\}$.
 %Therefore by Remark \ref{SOTconvexarg}, we obtain the desired net by taking convex combination. 
\end{proof}

The following corollary is the numerical radius version of \cite[VI.4. Proposition 4.11]{HWW}. We omit the proof since it can be proved by the same way using Proposition \ref{scap}.

\begin{cor}\label{Mebd}
Let $X$ be a reflexive Banach space with $n(X)>0$. If $\mathcal{K}_v(X)$ is an M-ideal in $\mathcal{L}_v(X)$, then $\mathcal{K}_v(X)^{**}=\mathcal{L}_v(X)$. In particular, $\mathcal{K}_v(X)$ is M-embedded.
\end{cor}

A Banach space $X$ is said to have the compact approximation property (CAP, in short) if the identity map can be approximated by compact operators in the strong operator topology. For a constant $\lambda>0$, if $X$ has the CAP with the approximation for the identity bounded by $\lambda$, then we say that $X$ has the $\lambda$-compact approximation property ($\lambda$-CAP, in short). 

Theorem \ref{tfae0}, the analogy for an equivalent norm $N$ of the equivalence in Theorem \ref{eqvalence11}, shows that if $\mathcal{K}_N(X)$ is an M-ideal in $\mathcal{L}_N(X)$, then $X$ has the $\lambda$-CAP for some $\lambda>0$. In Proposition \ref{scap}, we show that if $\mathcal{K}_v(X)$ is an M-ideal in $\mathcal{L}_v(X)$, then $X$ and $X^*$ both have the $n(X)^{-1}$-CAP whenever $n(X)>0$. 

\begin{cor}Let $X$ be a Banach space with $n(X)>0$. If $\mathcal{K}_v(X)$ is an M-ideal in $\mathcal{L}_v(X)$, then $X$ and $X^*$ have the $n(X)^{-1}$-CAP. 
\end{cor}
Note that the well-known Enflo-Davie subspace $Z$ of $\ell_p$ for $2<p<\infty$ (or of $c_0$) does not have the CAP (see \cite{LT}). Therefore, by Proposition \ref{scap}, $\mathcal{K}_v(Z)$ cannot be an M-ideal in $\mathcal{L}_v(Z)$.

~

\subsection{The characterization of M-ideals and the numerical index}~

In order to give the aforementioned characterization of the M-ideal of compact operators, the analogy of Theorem \ref{eqvalence11}, we define the essential-$N$-norm $N_e(T)$ of $T$ for the equivalent norm $N$ on $\mathcal{L}(X)$ as follows. 
$$N_e(T)=\inf_{K\in\mathcal{K}_N(X)}N(T-K).$$

\begin{theorem}\label{tfae0}
For a Banach space $X$ and a norm $N$ on $\mathcal{L}(X)$ which is equivalent to the operator norm, the following are equivalent.
\begin{enumerate}
\item $\mathcal{K}_N(X)$ is an M-ideal in $\mathcal{L}_N(X)$.
\item For all $T\in\mathcal{L}_N(X)$, there is a bounded shrinking compact approximation $\{K_\alpha\}$ of $T$ such that
\begin{align*}
\limsup_\alpha N(S+T-K_\alpha)\leq\max\{N(S),N_e(S)+N(T)\} \text{~for every~} S\in\mathcal{L}_N(X).
\end{align*} 
\item For all $T\in\mathcal{L}_N(X)$, there is a bounded shrinking compact approximation $\{K_\alpha\}$ of $T$ such that
\begin{align*}
\limsup_\alpha N(S+T-K_\alpha)\leq\max\{N(S),N(T)\} \text{~for every~} S\in\mathcal{K}_N(X).
\end{align*}
\end{enumerate}
\end{theorem}

\begin{proof}
Since the equivalence is obtained by the similar arguement in \cite[Theorem 3.1]{Wer}, we omit the proof. We just mention that the condition `boundedness' of the net $(K_\alpha)$ is followed by \cite[Proposition 2.3 and Lemma 2.4]{Wer}, and the condition `shrinking' of the net $(K_\alpha)$ can be obtained by the convex combination arguement in the proof of Proposition \ref{scap}.
\end{proof}

In fact, in the proof of Example \ref{basicexample}, it is proved that $\mathcal{L}_v(\ell_p(F_i))$ and $\mathcal{L}_v(c_0(F_i))$  satisfy the condition (3) of Theorem \ref{tfae0}. Indeed, for each $T\in \mathcal{L}_v(\ell_p(F_i))$ (or $\mathcal{L}_v(c_0(F_i))$) and the canonical projection $P_j$ from $\ell_p(F_i)$ to the first $j$th-coordinates, a sequence $\{K_n\}$ defined by $K_n=P_nT+TP_n-P_nTP_n$ is a shrinking compact approximation for $T$ such that
$$\limsup_{n\rightarrow \infty} v(S+T-K_n)\leq\max\{v(S),v(T)\} \text{~for every~} S\in\mathcal{K}_v(\ell_p).$$

In \cite{K1}, Kalton introduced property $(M)$ in order to characterize a separable Banach space $X$ such that $\mathcal{K}(X)$ is an M-ideal in $\mathcal{L}(X)$. The characterization for general Banach spaces is also given in terms of the net-version of property $(M)$ as follows (see \cite[Chap. VI]{HWW}).

 We say that a Banach space $X$ has property $(M)$ if $(x,y)\in X\times X$ satisfies $\|x\|\leq\|y\|$, then $$\limsup_\beta\|x+x_\beta\|\leq \limsup_\beta\|y+x_\beta\|$$
for every bounded weakly-null net $\{x_\beta\}$ in $X$. We say that a Banach space $X$ has property $(M^*)$ if $(x^*,y^*)\in X^*\times X^*$ satisfies $\|x^*\|\leq\|y^*\|$, then $$\limsup_\beta\|x^*+x^*_\beta\|\leq \limsup_\beta\|y^*+x^*_\beta\|$$
for every bounded weakly-$*$-null net $\{x^*_\beta\}$ in $X^*$.

For an infinite dimensional Banach space $X$, it is known that $\mathcal{K}(X)$ is an M-ideal in $\mathcal{L}(X)$ if and only if $X$ has property $(M)$ (can be replaced by property $(M^*)$) and there is a shrinking compact approximation $\{K_\alpha\}$ for the identity such that $\lim_\alpha \|id_X-2K_\alpha\|=1$. We show that if $\mathcal{K}_v(X)$ is an M-ideal in $\mathcal{L}_v(X)$, then $X$ satisfies some weaker conditions which are similar to properties $(M)$ and $(M^*)$.

\begin{prop}\label{Mineq}
Let $X$ be an infinite dimensional Banach space with $n(X)>0$. If $\mathcal{K}_v(X)$ is an M-ideal in $\mathcal{L}_v(X)$, then the following holds.
\begin{enumerate}
\item If $(x,y)\in X\times X$ satisfies $\|x\|\leq\|y\|$, then
$$n(X)\limsup_\beta\|x+x_\beta\|\leq \limsup_\beta\|y+x_\beta\|$$
for every bounded weakly-null net $\{x_\beta\}$ in $X$.
\item If $(x^*,y^*)\in X^*\times X^*$ satisfies $\|x^*\|\leq\|y^*\|$, then
$$n(X)\limsup_\beta\|x^*+x^*_\beta\|\leq \limsup_\beta\|y^*+x^*_\beta\|$$
for every bounded weakly-$*$-null net $\{x^*_\beta\}$ in $X^*$.
\end{enumerate}
\end{prop}
\begin{proof} 

 Since the proofs of (1) and (2) are very similar, we only give the one for (2). Without loss of generality, we assume that $y^*\neq 0$. By Theorem \ref{tfae0}, there is a shrinking compact approximation $\{K_\alpha\}$ for the identity $id_X$ such that 
$$\limsup_\alpha v(S+id_X-K_\alpha)\leq1\text{~for every~}S\in B_{\mathcal{K}_v(X)}.$$

For $\varepsilon\in (0,1)$, take $x\in \|y^*\|^{-1}B_{X}$ such that $y^*(x)>1-\varepsilon$. Then, for any $\alpha$, we have
\begin{align*}
\limsup_\beta\|x^*+x^*_\beta\|&\leq \limsup_\beta\|y^*(x) x^*+x^*_\beta\|+|1-y^*(x)|\|x^*\|\\
&\leq \limsup_\beta\|(y^*+x^*_\beta)(x) x^*+(id^*_X-K^*_\alpha)x^*_\beta\|+\varepsilon\|x^*\|\\
&\leq \limsup_\beta\|(x\otimes x^*)(y^*+x^*_\beta)+(id^*_X-K^*_\alpha)(y^*+x^*_\beta)\|+\|(id^*_X-K^*_\alpha)y\|+\varepsilon\|x^*\|\\
&\leq\|x\otimes x^*+id^*_X-K^*_\alpha\|\limsup_\beta\|y^*+x^*_\beta\|+\|(id^*_X-K^*_\alpha)y^*\|+\varepsilon\|x^*\|\\
&=\|x^*\otimes x+id_X-K_\alpha\|\limsup_\beta\|y^*+x^*_\beta\|+\|(id^*_X-K^*_\alpha)y^*\|+\varepsilon\|x^*\|\\
&\leq n(X)^{-1}v(x^*\otimes x+id_X-K_\alpha)\limsup_\beta\|y^*+x^*_\beta\|+\|(id^*_X-K^*_\alpha)y^*\|+\varepsilon\|x^*\|.
\end{align*}
Since $x^*\otimes x$ is a compact operator on $X$ and it satisfies $v(x^*\otimes x)\leq\|x^*\otimes x\|\leq1$, we have the following.
\begin{align*}
\limsup_\beta\|x^*+x^*_\beta\|\leq  n(X)^{-1}\limsup_\beta\|y^*+x^*_\beta\|+\|(id^*_X-K^*_\alpha)y^*\|+\varepsilon\|x^*\|.
\end{align*}
Since $\alpha$ and $\varepsilon$ are chosen arbitrarily, we get 
\begin{align*}
\limsup_\beta\|x^*+x^*_\beta\|\leq  n(X)^{-1}\limsup_\beta\|y^*+x^*_\beta\|.
\end{align*}
\end{proof}

It is worth to recall property $M(r,s)$ in \cite{CN} given for $r,s\in (0,1]$ which is a generalization of property $(M)$. We say that a Banach space $X$ has property $M(r,s)$ if $(x,y)\in X\times X$ satisfies $\|x\|\leq\|y\|$, then $$\limsup_\beta\|rx+sx_\beta\|\leq \limsup_\beta\|y+x_\beta\|$$ for every bounded weakly-null net $\{x_\beta\}$ in $X$. We say that a Banach space $X$ has property $M^*(r,s)$ if $(x^*,y^*)\in X^*\times X^*$ satisfies $\|x^*\|\leq\|y^*\|$, then $$\limsup_\beta\|rx^*+sx^*_\beta\|\leq \limsup_\beta\|y^*+x^*_\beta\|$$for every bounded weakly-$*$-null net $\{x^*_\beta\}$ in $X^*$. These are generalizations of properties $(M)$ and $(M^*)$, and introduced to investigate $M(r,s)$-ideals of compact operators. They share similar features with properties $(M)$ and $(M^*)$ (see \cite{HJO} for details), and we list a few facts. First, if $X$ has property $M^*(r,s)$, then it has property $M(r,s)$ and satisfies the $M(r,s)$-inequality (see \cite{HJO} and \cite{O} respectively). Secondly, if $X$ satisfies the $M(r,s)$-inequality for some  $r,s\in (0,1]$ such that $r+s>1$, then $X$ is an Asplund space \cite{CNO}. 

Since the conditions (1) and (2) of Proposition \ref{Mineq} are the same with properties $M(n(X),n(X))$ and $M^*(n(X),n(X))$ respectively, we see that (2) implies (1). Moreover, we have the following.

\begin{cor}\label{RNPMideal} Let $X$ be a Banach space with $n(X)>1/2$. If $\mathcal{K}_v(X)$ is an M-ideal in $\mathcal{L}_v(X)$, then $X$ is Asplund. In particular, if $X$ is a Banach space containing an isomorphic copy of $\ell_1$ such that $n(X)>1/2$, then $\mathcal{K}_v(X)$ is not an M-ideal in $\mathcal{L}_v(X)$.
\end{cor} 

We don't know whether the same statement in Corollary \ref{RNPMideal} holds or not for a Banach space $X$ with $n(X)>0$, but the version of the operator norm is well known. In \cite{L2}, it was firstly proved that if $X$ is a Banach space containing an isomorphic copy of $\ell_1$, then $\mathcal{K}(X)$ is not an M-ideal in $\mathcal{L}(X)$, and later, it was shown that if $\mathcal{K}(X)$ is an M-ideal in $\mathcal{L}(X)$, then $X$ is Asplund (see \cite[Corollary 4.5]{HWW}). 

\begin{question} Let $X$ be a Banach space with $n(X)>0$. Is $X$ Asplund whenever $\mathcal{K}_v(X)$ is an M-ideal in $\mathcal{L}_v(X)$?
\end{question}

To investigate an inifinite dimensional Banach space $X$ such that $\mathcal{K}_v(X\oplus_\infty X)$ or $\mathcal{K}_v(X\oplus_1 X)$ is an M-ideal, we recall the modulus  $\bar{\delta}_X$ of asymptotic uniform convexity and the modulus $\bar{\rho}_X$ of asymptotic uniform smoothness defined for $t\in[0,\infty)$ as

\begin{align*}
\bar{\delta}_X(t)&= \inf_{x\in S_X} \bar{\delta}_X(x,t) \text{~where~} \bar{\delta}_X(x,t)=\sup_{\dim(X/Y)<\infty}\inf_{y\in Y,\|y\|\geq t}\lVert x+y\rVert-1\\
\bar{\rho}_X(t)&=\sup_{x\in S_X} \bar{\rho}_X(x,t) \text{~where~} \bar{\rho}_X(x,t)=\inf_{\dim(X/Y)<\infty}\sup_{y\in Y,\|y\|\leq t}\lVert x+y\rVert-1.
\end{align*}

For a dual space $X$, the different moduli $\bar{\delta}^*_X$ and $\bar{\rho}^*_X$ are considered by taking all finite-codimensional weak-$*$-closed subspaces $Y$ of $X$. These moduli had been firstly introduced in \cite{Mil}, and it is investigated intensively. We refer the reader to \cite{JLDS} for fruitful results on these moduli, and we present only the well-known inequalities for further use (see, for example, \cite[Proposition 2.2]{GP}).

$$ \lVert x\rVert+\lVert x\rVert\bar{\delta}_X\left(\frac{\limsup_\alpha\lVert x_\alpha\rVert}{\lVert x\rVert}\right) \leq \limsup_\alpha\lVert x+x_\alpha\rVert \leq \lVert x\rVert+\lVert x\rVert\bar{\rho}_X\left(\frac{\limsup_\alpha\lVert x_\alpha\rVert}{\lVert x\rVert}\right)$$
for a non-zero vector $x\in X$ and a bounded weakly-null net $\{x_\alpha\}\subset X$. In case that $X$ is a dual space, the same inequalities hold for $\bar{\delta}^*_X$ and $\bar{\rho}^*_X$ with bounded weakly-$*$-null net $\{x_\alpha\}$.

It is worth to remark that $\mathcal{K}\left(X\oplus_1 X\right)$ is not an M-ideal for any infinite dimensional Banach space $X$, and $X$ has property $(m_\infty)$ whenever  $\mathcal{K}\left(X\oplus_\infty X\right)$ is an M-ideal (see \cite[Chap. VI]{HWW}). We say that $X$ has property $(m_\infty)$ if  $$\limsup_\beta\|x+x_\beta\|=\max \{\|x\|,\limsup_\beta\|x_\beta\|\}$$
for every $x\in X$ and every bounded weakly-null net $\{x_\beta\}\subset X$. 

\begin{cor}\label{modulus}
Let $X$ be an infinite dimensional Banach space such that $n(X)>0$.
\begin{enumerate}
\item If $\mathcal{K}_v(X\oplus_\infty X)$ is an M-ideal in $\mathcal{L}_v(X\oplus_\infty X)$, then it holds that 
$$n(X)\leq \min\left\{ \frac{1}{1+\bar{\delta}_X(1)},\frac{1+\bar{\rho}_{X^*}(1)}{2}\right\}.$$
\item If $\mathcal{K}_v(X\oplus_1 X)$ is an M-ideal in $\mathcal{L}_v(X\oplus_1 X)$, then it holds that 
$$n(X)\leq\min \left\{\frac{1+\bar{\rho}_X(1)}{2},\frac{1}{1+\bar{\delta}_{X^*}(1)}\right\}.$$
\end{enumerate}
\end{cor}
\begin{proof} We only prove (1) since the rest can be proved in a similar manner.

Note first that  $n(X\oplus_\infty X)=n(X)>0$ (see \cite[Proposition 1]{MP}). Since $X$ is infinite dimensional, there are a weakly-null net $\{x_\alpha\}\subset S_X$ and a weakly-$*$-null net $\{x^*_\beta\}\subset S_{X^*}$. Clearly, we see that $\|(x,x_\alpha)\|_\infty=1$ and $\|(x^*,x^*_\beta)\|_1=2$ for arbitrary $x\in S_X$ and $x^*\in S_{X^*}$ where $\|\cdot\|_\infty$ and $\|\cdot\|_1$ are norms on each direct sums $X\oplus_\infty X$ and $X^*\oplus_1 X^*$.

Hence, from Proposition \ref{Mineq}, we obtain
\begin{align*}
n(X)&\leq \frac{\limsup_\alpha \|(x,0)+(0,x_\alpha)\|_\infty}{\limsup_\alpha \|(0,x)+(0,x_\alpha)\|_\infty}\leq \frac{1}{1+\bar{\delta}_X(1)} \quad\text{~and~}\\
 n(X)&\leq \frac{\limsup_\alpha \|(0,x^*)+(0,x^*_\beta)\|_1}{\limsup_\alpha \|(x^*,0)+(0,x^*_\beta)\|_1}\leq \frac{1+\bar{\rho}_{X^*}(1)}{2}.
\end{align*}
\end{proof}

We don't know whether there is an infinite dimensional Banach space $X$ such that $\mathcal{K}_v\left(X\oplus_1 X\right)$ is an M-ideal or not, but we get the following consequence of Corollary \ref{modulus}.
\begin{example}
For a sequence of finite dimensional Banach spaces $\{F_i\}$ and a Banach space $X=c_0(F_i)$, if $\inf_in(F_i)>\frac{1}{2}$, then $\mathcal{K}_v\left(X\oplus_1 X\right)$ is not an M-ideal in $\mathcal{L}_v\left(X\oplus_1 X\right)$.
\end{example}
\begin{proof}  From the equality $\bar{\delta}_{X}(1)=\bar{\rho}_{X}(1)=0$ which is obvious, we get $n(X)\leq\frac{1}{2}$ whenever $\mathcal{K}_v\left(X\oplus_1 X\right)$ is an M-ideal by Corollary \ref{modulus}. This contradicts to $n(X)=\inf_in(F_i)>\frac{1}{2}$ (see \cite[Proposition 1]{MP}).
\end{proof}

$~$

\subsection{Proximinality and Farthest points}~

A subspace $J$ of $X$ is said to be proximinal in $X$ if for each element $x\in X$ there exists $j\in J$ such that $\|x-j\|=d(x,J)$.
Since it is well known that every M-ideal is proximinal, we see that $\mathcal{K}_v(\ell_p)$ is proximinal in $\mathcal{L}_v(\ell_p)$ for $p\in(1,\infty)$ whenever $n(\ell_p)>0$. 

In \cite[Corollary 3.2]{Wer}, the author derived a stronger result than the proximinality using the characterization (2) in Theorem \ref{eqvalence11} of a Banach space $X$ for which $\mathcal{K}(X)$ is an M-ideal in $\mathcal{L}(X)$. We obtain a similar observation in the setting of equivalent norms. 

\begin{cor}\label{WerCor}
Let $X$ be a Banach space such that at least one of $X^*$ and $X^{**}$ has the Radon-Nikod\'ym property, and $N$ be an equivalent norm on $\mathcal{L}(X)$. Suppose that $\mathcal{K}_N(X)$ is an M-ideal in $\mathcal{L}_N(X)$. If $T\in\mathcal{L}_N(X)$ and a net $\{T_\alpha\}\subset\mathcal{K}_N(X)$ satisfy that $T_\alpha^*$ converges to $T^*$ strongly, then there exists $K\in\overline{co}\{T_\alpha\}$ such that $N_e(T)=N(T-K)$.
\end{cor}
\begin{proof}
Since the proof is almost the same with that of \cite[Corollary 3.2]{Wer}, we only mention the difference. In the proof of \cite[Corollary 3.2]{Wer}, it is used that the weak operator topology coincides with the weak topology on $\mathcal{K}(X)$. This is due to the facts that if $\mathcal{K}(X)$ is an M-ideal in $\mathcal{L}(X)$, then $X^*$ has the Radon-Nikod\'ym property (RNP, in short) (see \cite[VI.4 Corollary 4.5]{HWW}),  and that if $X^*$ (or $X^{**}$) has the RNP, then it holds that 
$$\mathcal{K}(X)^*=\overline{\text{span}}\{x^{**}\otimes x^*: x^{**}\in X^{**}, x^*\in X^*\} \text{ (see \cite[Theorem 1]{FS})}.$$
To the best of our knowledge, it is not known whether $X^*$ (or $X^{**}$) has the RNP or not when $\mathcal{K}_N(X)$ is an M-ideal in $\mathcal{L}_N(X)$ for an equivalent norm $N$ on $\mathcal{L}(X)$. Therefore, we need to assume that $X^*$ or $X^{**}$ has RNP to get the same result. The rest is the same with the proof  \cite[Corollary 3.2]{Wer}.
\end{proof}

Thanks to Corollary \ref{RNPMideal}, we get the following.
\begin{cor}\label{WerCor}
Let $X$ be a Banach space with $n(X)>\frac{1}{2}$ and suppose that $\mathcal{K}_v(X)$ is an M-ideal in $\mathcal{L}_v(X)$. If $T\in\mathcal{L}_v(X)$ and a net $\{T_\alpha\}\subset\mathcal{K}_v(X)$ satisfy that $T_\alpha^*$ converges to $T^*$ strongly, then there exists $K\in\overline{co}\{T_\alpha\}$ such that $v_e(T)=v(T-K)$.
\end{cor}

Let $X$ be a normed space with the norm $N$ and $C$ be a subset of $X$. An element $x\in X$ is said to have a farthest point $y\in C$ if it satisfies that 
$$N(x-y)=\rho_N(x,C)$$
where $\rho_N(x,C)=\sup\{N(x-z):z\in C\}$ which is  the farthest distance from $x$ to $C$ with respect to the norm $N$.

In \cite[Lemma 2.19 and Proposition 2.23]{BLR2}, the authors studied the existence of farthest points in $B_{\mathcal{K}(X)}$ using norm attaining operators. We get similar results with numerical radius attaining operators.

In order to do this, we get a form of an extreme point in $B_{\mathcal{K}_v(X)^*}$ using the following lemma.

\begin{lemma}\label{extremepoints}\cite[I.1 Lemma 1.5]{HWW}
If a closed subspace $J$ of a Banach space $X$ is an M-ideal, then it holds that 
$$\text{Ext} (B_{X^*})=\text{Ext} (B_{J^*})\cup \text{Ext} (B_{J^\perp}).$$
\end{lemma}

\begin{prop}\label{expt}
For a Banach space $X$ with $n(X)>0$, if $\mathcal{K}_v(X)$ is an M-ideal in $\mathcal{L}_v(X)$, then 
\begin{enumerate}
\item $\text{Ext} (B_{\mathcal{L}_v(X)^*})\subset \overline{\{\lambda x^*\otimes x:(x^*,x)\in\Pi(X)\text{ and }\lambda\in\mathbb{T}\}}^{w^*}$, and 
\item $\text{Ext}(B_{\mathcal{K}_v(X)^*})\subset\{\lambda x^{**}\otimes x^*:(x^{**},x^*)\in\Pi(X^*)\text{ and }\lambda\in\mathbb{T}\}$
\end{enumerate}
where $\mathbb{T}$ is the set of scalars of modulus $1$. 
\end{prop}
\begin{proof}~
For each $T\in\mathcal{L}_v(X)$, it holds that $v(T)=\sup\{\re \lambda x^*(Tx)~:~(x^*,x)\in\Pi(X)\text{ and }\lambda\in\mathbb{T}\}$. Hence, from the Hahn-Banach separation theorem, we have 
$$\overline{\text{conv}}^{w^*}\left(\{\lambda x^*\otimes x:(x^*,x)\in\Pi(X)\text{ and }\lambda\in\mathbb{T}\}\right)=B_{\mathcal{L}_v(X)^*},$$
and this shows that
\begin{equation*}\label{Lext}
\text{Ext} (B_{\mathcal{L}_v(X)^*})\subset \overline{\{\lambda x^*\otimes x:(x^*,x)\in\Pi(X)\text{ and }\lambda\in\mathbb{T}\}}^{w^*}    
\end{equation*}
from Milman theorem.

For fixed $\Psi\in \text{Ext}  B_{\mathcal{K}_v(X)^*}$, from Lemma \ref{extremepoints} which implies $\text{Ext} B_{\mathcal{K}_v(X)^*}\subset \text{Ext}  B_{\mathcal{L}_v(X)^*}$, we get a net $\Psi_\beta=\lambda_\beta x_\beta^*\otimes x_\beta$ for $(x_\beta^*,x_\beta)\in\Pi(X)$ and $\lambda_\beta\in\mathbb{T}$ such that $\Psi_\beta(T)=\lambda_\beta x_\beta^* (T x_\beta)$ converges to $\Psi(T)$ for every $T\in \mathcal{L}_v(X)$. Passing to appropriate subnets, we may assume that $x_\beta^*$ and $x_\beta$ converge weakly-$*$ to $x^*\in B_{X^*}$ and $x^{**}\in B_{X^{**}}$ respectively, and $\lambda_\beta$ converges to $\lambda\in\mathbb{T}$. Since it holds that
$$\Psi(K)=\lim_\beta \lambda_\beta K^*(x_\beta^*)(x_\beta)=\lambda x^{**}(K^*x^*)$$
for every $K\in\mathcal{K}_v(X)$, we see that $\Psi=\lambda x^{**}\otimes x^*$ on $\mathcal{K}_v(X)$.

From the proof of Proposition \ref{scap}, take a shrinking compact approximation  $\{K_\alpha\}\subset \mathcal{K}_v(X)$  for the identity $id_X$ which converges to $id_X$ in the  $\sigma(\mathcal{L}_v(X),\mathcal{K}_v(X)^*)$-topology.

We have that 
$$|x^{**}(x^*)|=\lim_\alpha|x^{**}(K_\alpha^*x^*)|=|\Psi(id_X)|=\lim_\beta|\Psi_\beta(id_X)|=\lim_\beta|x_\beta^*(x_\beta)|=1.$$

Thus, we deduce that $(\mu x^{**},x^*)\in\Pi(X^*)$ for some $\mu\in\mathbb{T}$ which completes the proof.
\end{proof}

\begin{theorem}\label{Farpt}
Let $X$ be a Banach space with $n(X)>0$.
\begin{enumerate}
\item For each $T\in\mathcal{L}_v(X)$, $\rho_v\left(T,B_{\mathcal{K}_v(X)}\right)=v(T)+1$.
\item If $T\in\mathcal{L}_v(X)$ attains its numerical radius, then $T$ has a farthest point in $B_{\mathcal{K}_v(X)}$.
\item If $\mathcal{K}_v(X)$ is an M-ideal in $\mathcal{L}_v(X)$ and $T\in\mathcal{L}_v(X)$ has a farthest point in $B_{\mathcal{K}_v(X)}$, then $T^*$ attains its numerical radius.
\end{enumerate}
\end{theorem}
\begin{proof}~

\begin{itemize}
\item [(1)]It is clear that $\rho_v\left(T,B_{\mathcal{K}_v(X)}\right)\leq v(T)+1$. 
For the converse, let $\varepsilon>0$ be given. For $(x^*,x)\in\Pi(X)$ and $\lambda\in\mathbb{T}$ such that $\bar{\lambda}x^*(Tx)=|x^*(Tx)|\geq v(T)-\varepsilon$, consider a rank $1$ operator $K=-\lambda x^*\otimes x$. We have
$$v(T-K)\geq|x^*(Tx)+\lambda|=\bar{\lambda}x^*(Tx)+1\geq v(T)+1-\varepsilon.$$
\item [(2)]
Suppose that $T$ attains its numerical radius at $(x_0^*,x_0)\in\Pi(X)$, and $\lambda_0\in\mathbb{T}$ satisfies $\bar{\lambda_0}x_0^*Tx_0=v(T)$. It is clear that $K_0=\lambda_0x^*_0\otimes x_0$ is a farthest point of $T$. 

\item [(3)] By the assumption, there exists $K\in B_{\mathcal{K}_v(X)}$ such that $v(T-K)=v(T)+1$. Since the set of extreme points of the closed unit ball of a dual space is a James boundary, there exists $\Psi\in \text{Ext}\left(B_{\mathcal{L}_v(X)^*}\right)$ which satisfies $\Psi(T-K)=v(T-K)$. Since $\mathcal{K}_v(X)$ is an M-ideal in $\mathcal{L}_v(X)$, we see that
$$\Psi\in \text{Ext}\left(B_{\mathcal{K}_v(X)^*}\right)~\text{ or }~ \Psi\in \text{Ext}\left(B_{\mathcal{K}_v(X)^\perp}\right)$$
from Lemma \ref{extremepoints}. If $\Psi$ is belonging to $\text{Ext}\left(B_{\mathcal{K}_v(X)^\perp}\right)$, then we have $v(T)+1=\Psi(T-K)=\Psi(T)$ which is not possible. Therefore, it holds that $\Psi\in \text{Ext}\left(B_{\mathcal{K}_v(X)^*}\right)$, and we get $\Psi=\mu x^{**}\otimes x^*$ for some $\mu\in\mathbb{T}$ and $(x^{**},x^*)\in\Pi(X^*)$ by Proposition \ref{expt}. Thus, we have
$$v(T)+1=\Psi(T-K)=\mu x^{**}(T^*x^*)-\mu x^{**}(K^*x^*)\leq v(T^*)+v(K^*)=v(T)+1.$$
Consequently, $T^*$ attains its numerical radius at $(x^{**},x^*)$.
\end{itemize}
\end{proof}

$~$

\subsection{Compact perturbation property}~

In \cite{WWerner}, it is shown that if $\mathcal{K}(X)$ is an M-ideal in $\mathcal{L}(X)$, then $X$ fulfills a condition that every $T\in\mathcal{L}(X)$ whose adjoint $T^*$ does not attain its norm satisfies $\|T\|=\|T\|_e$. Very recently, this condition has come to be called the adjoint compact perturbation property (in short ACPP), and many interesting results have been found (see \cite{AGPT,GP,HK,JMR}). The main aim of the present section is to get the analogous result for $X$ such that $\mathcal{K}_v(X)$ is an M-ideal in $\mathcal{L}_v(X)$.

\begin{theorem}\label{NACPP}
Let $X$ be a Banach space with $n(X)>0$ such that $\mathcal{K}_v(X)$ is an M-ideal in $\mathcal{L}_v(X)$.
\begin{enumerate}
\item For $T\in\mathcal{L}_v(X)$, if $T^*$ does not attain its numerical radius, then we have
$$v(T)=v_e(T).$$
\item The set of operators whose adjoint operators do not attain their numerical radius is nowhere dense in $\mathcal{L}_v(X)$.
\end{enumerate}
\end{theorem}
\begin{proof}~

\begin{enumerate}
\item Since $$H=\{\xi \in \mathcal{L}_v(X)^*~:~\xi(T)=v(T)\}$$
 is a weak-$*$-closed supporting manifold of $B_{\mathcal{L}_v(X)^*}$, there exists $\Psi\in \text{Ext} B_{\mathcal{L}_v(X)^*}$ such that $\Psi(T)=v(T)$.
From Lemma \ref{extremepoints}, we see that $\Psi$ is an element of either $\text{Ext} B_{\mathcal{K}_v(X)^*}$ or $\text{Ext} B_{\mathcal{K}_v(X)^\perp}$. In case $\Psi\in \text{Ext} B_{\mathcal{K}_v(X)^*}$, we see that $T^*$ attains its numerical radius. Indeed, 
$$v(T^*)=v(T)=\Psi(T)=\lambda x^{**}(T^*x^*)$$
for some $\lambda\in\mathbb{T}$ and $(x^{**},x^*)\in\Pi(X^*)$ by (2) of Proposition \ref{expt}. Therefore, we have that $\Psi\in \text{Ext} B_{\mathcal{K}_v(X)^\perp}$ which shows 
$$v(T)=\Psi(T)\leq v_e(T)\leq v(T).$$
\item Note that a set 
$$E=\{S\in \mathcal{L}_v(X)~:~v(S)=v_e(S)\}$$
is closed and contains all the operators whose adjoint does not attain its numerical radius by (1). Hence, it is enough to show that $E$ is nowhere dense. Take arbitrary $T\in E$ and $\varepsilon>0$. In case $T=0$, we see that $U=T+\frac{\varepsilon}{2}x^*\otimes x$ satisfies $$v(U-T)\leq \left\|\frac{\varepsilon}{2}x^*\otimes x\right\|<\varepsilon\text{~and~}v(U)>0=v_e(T)=v_e(U)$$ for arbitrary $(x^*,x)\in \Pi(X)$.

Otherwise, take $(x^*,x)\in \Pi(X)$ such that $|x^*(Tx)|>v(T)-\min\{1,v(T)\}\frac{\varepsilon}{4}$. We see that $U=T+\frac{x^*(Tx)}{|x^*(Tx)|}\frac{\varepsilon}{2}x^*\otimes x$ satisfies $$v(U-T)\leq \left\|\frac{\varepsilon}{2}x^*\otimes x\right\|<\varepsilon\text{~and~}v(U)\geq|x^*(Tx)|+\frac{\varepsilon}{2}> v(T)+\frac{\varepsilon}{4}>v_e(T)=v_e(U).$$

\end{enumerate}
\end{proof}

 We also give a description of the essential numerical radius in terms of a weakly-null net and a weakly-$*$-null net whenever the space of compact operators endowed with the numerical radius norm is an M-ideal. This is an analogy of \cite[VI.4. Proposition 4.7]{HWW} which is the version for the operator norm.
\begin{prop}
Let $X$ be an infinite-dimensional Banach space such that $n(X)>0$ and $w_v(T)$ for $T\in \mathcal{L}_v(X)$ be defined by
$$w_v(T)=\sup\{\limsup_\alpha|x^*_\alpha Tx_\alpha|:(x_\alpha^*,x_\alpha)\in\Pi(X)\text{~is a net satisfying }x_\alpha\overset{\text{$w$}}{\longrightarrow}0\text{ or }x_\alpha^*\overset{\text{$w^*$}}{\longrightarrow}0 \}.$$
If $\mathcal{K}_v(X)$ is an M-ideal in $\mathcal{L}_v(X)$, then it holds that
$$v_e(T)=w_v(T)$$
\end{prop}
\begin{proof}
Since the inequality $v_e(T)\geq w_v(T)$ is obvious for every $T\in\mathcal{L}_v(X)$, we only need to to prove the converse.

 From Hahn-Banach separation theorem, we see that there exists $\psi\in B_{\mathcal{K}_v(X)^\perp}$ such that $\psi(T)=v_e(T)$. Hence, the set 
$$H=\{\xi \in \mathcal{K}_v(X)^\perp~:~\xi(T)=v_e(T)\}$$
 is a weak-$*$-closed supporting manifold of $B_{\mathcal{K}_v(X)^\perp}$ which guarantees that $H\cap  \text{Ext} B_{\mathcal{K}_v(X)^\perp}$ is non-empty. Take $\Psi\in \text{Ext} B_{\mathcal{K}_v(X)^\perp}\cap H$. 

From Lemma \ref{extremepoints} and (1) of Proposition \ref{expt}, we see that $\Psi$ belongs to $\text{Ext} B_{\mathcal{L}(X)^*}$ and that there is a net $\Psi_\beta=\lambda_\beta x_\beta^*\otimes x_\beta$ for $(x_\beta^*,x_\beta)\in\Pi(X)$ and $\lambda_\beta\in\mathbb{T}$ such that $\Psi_\beta(L)=\lambda_\beta x_\beta^* (L x_\beta)$ converges to $\Psi(L)$ for every $L\in \mathcal{L}_v(X)$. Passing to appropriate subnets, we may assume that $x_\beta^*$ and $x_\beta$ converge weakly-$*$ to $x^*\in B_{X^*}$ and $x^{**}\in B_{X^{**}}$ respectively, and $\lambda_\beta$ converges to $\lambda\in\mathbb{T}$.  For any fixed $K\in\mathcal{K}_v(X)$, we observe that
$$0=\Psi(K)=\lim_\beta \Psi_\beta(K)=\lim_\beta \lambda_\beta K^*(x^*_\beta)(x_\beta)=\lambda x^{**}(K^*x^*).$$
Since $K$ is chosen arbitrarily, we see that at least one of $x^*$ and $x^{**}$ is the zero vector. Note that $x_\beta$ converges to $0$ weakly whenever $x^{**}=0$. Furthermore, we have
$$v_e(T)=\Psi(T)=\lim_\beta \Psi_\beta(T)=\lim_\beta \lambda x^*_\beta(Tx_\beta)\leq w_v(T).$$
\end{proof}

Finally, we present the space of compact operators which are not M-ideals endowed with the numerical radius norm using Theorem \ref{NACPP}.

\begin{prop}\label{counterexamnum}
   Let $X$ and $Y$ be Banach spaces over the scalar field $\mathbb{K}$ and assume that both $n(X\oplus_p\mathbb{K})$ and $n(Y\oplus_q\mathbb{K})$ are positive for $1\leq q<p<\infty$. If there exists an operator $T\in {\mathcal{L}(X,Y)}$ such that $T^*$ does not attain its norm, then the following hold.
    \begin{itemize}
        \item[(1)] $\mathcal{K}_v\left((X\oplus_p\mathbb{K})\oplus_\infty(Y\oplus_q\mathbb{K})\right)$ is not an M-ideal in $\mathcal{L}_v\left((X\oplus_p\mathbb{K})\oplus_\infty(Y\oplus_q\mathbb{K})\right)$.
        \item[(2)] $\mathcal{K}_v\left((X\oplus_p\mathbb{K})\oplus_1(Y\oplus_q\mathbb{K})\right)$ is not an M-ideal in $\mathcal{L}_v\left((X\oplus_p\mathbb{K})\oplus_1(Y\oplus_q\mathbb{K})\right)$.
    \end{itemize}
\end{prop}
\begin{proof}~

\begin{itemize}
\item[(1)] For an operator $S\in \mathcal{L}_v\left((X\oplus_p\mathbb{K})\oplus_\infty(Y\oplus_q\mathbb{K})\right)$ defined by 
$$S\left((x,\alpha),(y,\beta)\right)=\left((0,0),(Tx,\|T\|\alpha)\right),$$ 
it is enough to prove that $v_e(S)<v(S)$ and $S^*$ does not attain its numerical radius from (1) of Proposition \ref{NACPP}.

We may assume that $\|T\|=1$.  Since $S^*$ satisfies
    $$\left(S^*((x^*,a),(y^*,b))\right),\left((x,c),(y,d)\right)=\left((T^* y^*,b),(0,0)\right)\left((x,c),(0,0)\right)$$ 
for each $a,b,c,d\in\mathbb{K}$, $x^*\in X^*$, $y\in Y$ and $y^*\in Y^*$, we deduce the following inequalities for the conjugate $p^*$ and $q^*$ of $p$ and $q$ respectively. \begin{itemize}
\item The case $q=1$ 
\begin{align*}\|S^*\|&=\sup_{y^*\in S_{Y^*}}\left\|S^*\left((0,0),\left(y^*,1\right)\right)\right\|=\sup_{y^*\in S_{Y^*}}\left\|\left(T^*y^*,1\right)\right\|\\&=\left\|\left(\|T^*\|,1\right)\right\|_{p^*}=2^{\frac{1}{p^*}}>1.\end{align*}

\item The case $q>1$
    \begin{align*}
    \|S^*\|&=\sup_{\substack{y^*\in S_{Y^*}\\ 0\leq|\lambda|\leq1}}\left\|S^*\left((0,0),\left(|\lambda|y^*,\left(1-|\lambda|^{q^*}\right)^{\frac{1}{q^*}}\right)\right)\right\|\\
    &=\sup_{\substack{y^*\in S_{Y^*}\\ 0\leq|\lambda|\leq1}}\left(|\lambda|^{p^*}\|T^*y^*\|^{p^*}+\left(1-|\lambda|^{q^*}\right)^{\frac{p^*}{q^*}}\right)^{\frac{1}{p^*}}\\
    &=\sup_{0\leq|\lambda|\leq1}\left(|\lambda|^{p^*}\|T^*\|^{p^*}+\left(1-|\lambda|^{q^*}\right)^{\frac{p^*}{q^*}}\right)^{\frac{1}{p^*}}\\
    &=\sup_{0\leq|\lambda|\leq1}\left\|\left(|\lambda|,\left(1-|\lambda|^{q^*}\right)^{\frac{1}{q^*}}\right)\right\|_{p^*}\\
    &>1.
    \end{align*}
\end{itemize}
From these inequalities, it is easy to see that the norm of $S$ is greater than $1$ and $S^*$ does not attain its norm.

    We claim $v(S)=\|S\|$. Note that $v(S)=v(S^*)$ (see \cite{Bol}). For given $\varepsilon>0$, take a norm $1$ element $(y^*,t) \in Y^*\oplus_{q^*}\mathbb{K}$ such that
    $$\|\left(T^*y^*,t\right)\|=\|S^*((0,0),(y^*,t))\|>\|S\|-\varepsilon.$$
    By Hahn-Banach theorem, we choose norm $1$ elements $(x^{**},\tilde{t}_1)\in X^{**}\oplus_p \mathbb{K}$ and $(y^{**},\tilde{t}_2)\in X^{**}\oplus_q \mathbb{K}$ such that
    \begin{align*} 
    (x^{**},\tilde{t}_1)\left((T^*y^*,t)\right)&=\|(T^*y^*,t)\| \text{ and } (y^{**},\tilde{t}_2)\left((y^*,t)\right)=1.
    \end{align*}

    Then we have $\left(\left((x^{**},\tilde{t}_1),(y^{**},\tilde{t}_2)\right),\left((0,0),(y^*,t)\right)\right)\in \Pi\left(\left((X\oplus_p\mathbb{K})\oplus_\infty(Y\oplus_q\mathbb{K})\right)^*\right)$ and the following inequality which shows the claim.
    $$v(S^*)\geq \left((x^{**},\tilde{t}_1),(y^{**},\tilde{t}_2)\right)\left(S^*((0,0),(y^*,t))\right) >\|S\|-\varepsilon.$$
    
    From this claim, we also see that $S^*$ does not attain its numerical radius since it does not attain its norm.

    In order to prove $v_e(S)<v(S)$, define $V\in\mathcal{K}_v\left((X\oplus_p\mathbb{K})\oplus_\infty(Y\oplus_q\mathbb{K})\right)$  by 
$$V((x,\alpha),(y,\beta))=((0,0),(0,\alpha)).$$ It is clear that $\|S-V\|=\|T\|=1$, and this implies
    $$v_e(S)\leq v(S-V)\leq \|S-V\|=1<v(S).$$

\item[(2)]
    Since the proof is similar to that of (1), we focus on the differences. The operator $S$ in the proof of (1) is also well defined on $(X\oplus_p\mathbb{K})\oplus_1(Y\oplus_q\mathbb{K})$. Moreover, $S^*$ does not attain its numerical radius and $S$ satisfies $\|S\|>\|T\|$. In order to show $v(S)=\|S\|$, for given $\varepsilon>0$, take $y^*\in Y^*$ and $t\in\mathbb{K}$ such that $\|y^*\|^{q^*}+|t|^{q^*}=1$ and 
    $$\|\left(T^*y^*,t\right)\|=\|S^*((0,0),(y^*,t))\|>\|S\|-\varepsilon.$$

Thanks to the well known Bishop-Phelps theorem which asserts that the set of norm attaining functionals is norm dense in the dual space, we take norm $1$ elements $(x^{**},\tilde{t}_1)\in X^{**}\oplus_p\mathbb{K}$ and $(x^{*},\tilde{t}_2)\in X^{*}\oplus_{p^*}\mathbb{K}$ such that 
    \begin{align*}
        (x^{**},\tilde{t}_1)\left((T^*y^*,t)\right)>\|S\|-\varepsilon \text{ and }(x^{**},\tilde{t}_1)\left((x^{*},\tilde{t}_2)\right)=1.
    \end{align*}
 Hence, we have $\left(\left((x^{**},\tilde{t}_1),(0,0)\right),\left((x^*,\tilde{t}_2),(y^*,t)\right)\right)\in\Pi\left(\left((X\oplus_p\mathbb{K})\oplus_1(Y\oplus_q\mathbb{K})\right)^*\right)$ and 
\begin{align*}
v(S^*)&\geq \left((x^{**},\tilde{t}_1),(0,0)\right)\left(S^*\left((x^*,\tilde{t}_2),(y^*,t)\right)\right)\\
&= (x^{**},\tilde{t}_1)\left((T^*y^*,t)\right) \\
&>\|S\|-\varepsilon.
\end{align*}
By the same arguement in (1), we see that $S^*$ does not attain its numerical radius and $v_e(S)<v(S)$ which finish the proof.
\end{itemize}
\end{proof}

For $s\in (1,\infty)$ and an operator $S\in \mathcal{L}(\ell_s)$ defined by $S\left(\sum_i \alpha_i e_i\right)=S\left(\sum_i  \left(1-\frac{1}{i}\right)\alpha_i e_i\right)$  where $\{e_i\}$ is the canonical basis of $\ell_s$, it is obvious that $S^*$ does not attain its norm. Hence, we have the following corollary of Proposition \ref{counterexamnum}. Note that the numerical index of a complex Banach space is greater than $0$.

\begin{cor}For $s\in (1,\infty)$ and $1\leq q<p<\infty$, if it holds that $n((\ell_s\oplus_p\mathbb{K}))>0$ and $n((\ell_s\oplus_q\mathbb{K}))>0$,  then the following holds.
    \begin{itemize}
        \item[(1)] $\mathcal{K}_v\left((\ell_s\oplus_p\mathbb{K})\oplus_\infty(\ell_s\oplus_q\mathbb{K})\right)$ is not an M-ideal in $\mathcal{L}_v\left((\ell_s\oplus_p\mathbb{K})\oplus_\infty(\ell_s\oplus_q\mathbb{K})\right)$.
        \item[(2)] $\mathcal{K}_v\left((\ell_s\oplus_p\mathbb{K})\oplus_1(\ell_s\oplus_q\mathbb{K})\right)$ is not an M-ideal in $\mathcal{L}_v\left((\ell_s\oplus_p\mathbb{K})\oplus_1(\ell_s\oplus_q\mathbb{K})\right)$.
    \end{itemize}
\end{cor}

%%\subsection*{Statements \& Declarations} 

\subsection*{Funding}
The first and second author were supported by the National Research Foundation of Korea(NRF) grant funded by the Korea government(MSIT) [NRF-2020R1C1C1A01012267].

%%\subsection*{Competing Interests} The authors have no relevant financial or non-financial interests to disclose.

%%\subsection*{Author Contributions} All authors contributed to the whole part of works together such as the study conception, design and writing. All authors read and approved the final manuscript.

%%\subsection*{Data availability} All data generated or analysed during this study are included in this published article (and its supplementary information files).


\begin{thebibliography}{99}


\bibitem{AE} E. M. Alfsen and E. G. Effros, \emph{Structure in real Banach spaces}, Ann. of Math. (2) \textbf{96} (1972), 98-173. \url{ https://doi.org/10.2307/1970895}

\bibitem{AGPT}R. M. Aron, D. Garc\'ia, D. Pellegrino, and E. V. Teixeira, \emph{Reflexivity and nonweakly-null maximizing sequences}, Proc. Amer. Math. Soc. \textbf{148} (2020), no.~2, 741-750. \url{https://doi.org/10.1090/proc/14728}

\bibitem{BLR2} P. Bandyopadhyay, B-L.Lin and T. S. S. R. K. Rao,  \emph{Ball remotal subspaces of {B}anach spaces}, Colloq. Math. \textbf{114} (2009), no.~1, 119-133. \url{http://eudml.org/doc/283714}

\bibitem{Bol} B. Bollob\'as, \emph{An extension to the theorem of Bishop and Phelps}, Bull. London Math. Soc. \textbf{2} (1970), 181–182. \url{https://doi.org/10.1112/blms/2.2.181}

\bibitem{BD1} F. F. Bonsall and J. Duncan, \emph{Numerical ranges of operators on normed spaces and of elements of algebras,}, London Math. Soc. Lecture Note Ser. \textbf{2} Cambridge University Press, London-New York (1971)
\bibitem{BD2} F. F. Bonsall and J. Duncan, \emph{Numerical ranges. {II}}, London Math. Soc. Lecture Note Ser. \textbf{10} Cambridge University Press, New York-London (1973)

\bibitem {CN} J. C. Cabello and E. Nieto, \emph{On M-type structures and the fixed point property}, Houston J. Math. \textbf{26} (2000), 549–560. \url{https://www.math.uh.edu/~hjm/Vol26-3.html}

\bibitem {CNO} J. C. Cabello, E. Nieto and E. Oja, \emph{On ideals of compact operators satisfying the  $M(r,s)$-inequality}, J. Math. Anal. Appl. \textbf{220} (1998), no. 1,  334-348. \url{https://doi.org/10.1006/jmaa.1997.5888}

\bibitem{Dix} J. Dixmier \emph{Les fonctionnelles lin\'eaires sur l'ensemble des op\'erateurs born\'es d'un espace de {H}ilbert}, Ann. of Math. (2) \textbf{51} (1950), 387-408. \url{https://doi.org/10.2307/1969331}

\bibitem{FS} M. Feder and P. Shaphar \emph{Spaces of compact operators and their dual spaces}, Israel J. Math. \textbf{21} (1975), no.~1, 38-49. \url{https://doi.org/10.1007/BF02757132}

\bibitem{GP} L. C. Garc\'{\i}a-Lirola and C. Petitjean, \emph{On the weak maximizing properties}, Banach J. Math. Anal. \textbf{15} (2021), no. 3, Paper No. 55. \url{https://doi.org/10.1007/s43037-021-00141-x}

\bibitem{HJO} R. Haller, M. Johanson and E. Oja, \emph{{$M(r,s)$}-ideals of compact operators}, Czechoslovak Math. J. \textbf{62(137)} (2012), no. 3, 673-693. \url{https://doi.org/10.1007/s10587-012-0059-9}


\bibitem{HK} M. Han and S.K. Kim, \emph{M-ideals of compact operators and norm attaining operators}, preprint. \url{https://arxiv.org/abs/2402.12070}

\bibitem{HL} P. Harmand and \AA. Lima, \emph{Banach spaces which are M-ideals in their biduals} Trans. Amer. Math. Soc. \textbf{283} (1984), no.1, 253-264. \url{https://doi.org/10.2307/2000001}

\bibitem{HWW} P. Harmand, D. Werner and W. Werner,  \emph{M -ideals in Banach spaces and Banach algebras}, Lecture Notes in Math., 1547 Springer-Verlag, Berlin, 1993, viii+387 pp. ISBN: 3-540-56814-X \url{}

\bibitem{JLDS} W. B. Johnson, J. Lindenstrauss, D. Preiss, and G. Schechtman, \emph{Almost {F}r\'{e}chet differentiability of {L}ipschitz mappings between infinite-dimensional {B}anach spaces}, Proc. London Math. Soc. (3) \textbf{84} (2002), no.~3, 711-746. \url{https://doi.org/10.1112/S0024611502013400}


\bibitem{JMR} M. Jung, G. Mart\'inez-Cervantes and A. Rueda Zoca, \emph{Rank-one perturbations and norm-attaining operators}, Math. Z. \textbf{306} (2024), no. 4, Paper No. 62, 11 pp. \url{https://doi.org/10.1007/s00209-024-03458-z} 


\bibitem{K1} N. J. Kalton, \emph{M -ideals of compact operators}, Illinois J. Math. \textbf{37} (1993), no. 1, 147-169. \url{https://doi.org/10.1215/ijm/1255987254}

\bibitem{L1} \AA. Lima, \emph{M-ideals of compact operators in classical Banach spaces}, Math. Scand. \textbf{44} (1979), 207-217. \url{https://doi.org/10.7146/math.scand.a-11804}

\bibitem{L2} \AA. Lima, \emph{On M-ideals and best approximation}, Indiana Univ. Math. J. \textbf{31} (1982), 27-36.  \url{https://doi.org/10.1512/iumj.1982.31.31004
}

\bibitem{LT} J. Lindenstrauss and L. Tzafriri, \emph{Classical {B}anach spaces. I}, Springer-Verlag, Berlin-New York, 1977.

\bibitem{MMP} M. Mart\'in, J Mer\'i and M. Popov \emph{On the numerical index of real $L_p(\mu)$-spaces} Israel J. Math. \textbf{184} (2011),183–192. \url{https://doi.org/10.1007/s11856-011-0064-y}


\bibitem{MP} M. Mart\'in and R. Pay\'a \emph{Numerical index of vector-valued function spaces}, Studia Math. \textbf{142} (2000), 269-280. \url{https://doi.org/10.4064/sm-142-3-269-280}

\bibitem{Mil} V. D. Milman, \emph{Geometric theory of {B}anach spaces. {II}. {G}eometry of  the unit ball}, Uspehi Mat. Nauk \textbf{26} (1971), no.~6(162), 73-149. \url{https://doi.org/10.1070/RM1971v026n06ABEH001273}

\bibitem{Oja1}E. Oja, \emph{{$M$}-ideals of compact operators are separably determined}, Proc. Amer. Math. Soc. \textbf{126} (1998), no.~9, 2747-2753. \url{https://doi.org/10.1090/S0002-9939-98-04600-0} 

\bibitem{O} E. Oja, \emph{Geometry of Banach spaces having shrinking approximations of the identity}, Trans. Amer. Math. Soc. \textbf{352} (2000), no.~6, 2801-2823. \url{https://doi.org/10.1090/S0002-9947-00-02521-6}

\bibitem{Wer} D. Werner \emph{{$M$}-Ideals and the ``basic inequality''}, J. Approx. Theory \textbf{76} (1994), no.~1, 21-30. \url{https://doi.org/10.1006/jath.1994.1002.}

\bibitem{WWerner} W. Werner,  \emph{Smooth points in some spaces of bounded operators}, Integral Equations Operator Theory \textbf{15} (1992), no. 3, 496-502. \url{https://doi.org/10.1007/BF01200332}


\end{thebibliography}
\end{document}